\documentclass[a4paper,11pt]{article}
\usepackage[utf8]{inputenc}
\usepackage[T1]{fontenc}

\usepackage{hyperref}
\hypersetup{colorlinks = true, linkcolor = blue, citecolor = blue, urlcolor = blue}
\usepackage{amsthm,amsmath}
\usepackage{tikz}
\usepackage{mathrsfs,amssymb,amsfonts} 
\usepackage{enumitem}
\usepackage{fullpage}
\usepackage[babel]{microtype}
\usepackage[english]{babel}
\usepackage[capitalise]{cleveref}
\usepackage{thmtools}
\usepackage{mathtools, comment}
\usepackage{amssymb}
\usepackage[nomath]{lmodern}
\usepackage{graphicx}
\usepackage{pgf,tikz,tkz-graph,subcaption}
\usetikzlibrary{arrows,shapes}
\usetikzlibrary{decorations.pathreplacing}
\usepackage{tkz-berge}
\usepackage{enumitem}
\usepackage[normalem]{ulem}

\newcommand*{\ceilfrac}[2]{\mathopen{}\left\lceil\frac{#1}{#2}\right\rceil\mathclose{}}

\newcommand*{\abs}[1]{\lvert #1\rvert}

\allowdisplaybreaks

\usepackage[margin=1in]{geometry}
\newtheorem{defi}{Definition}
\newtheorem{conj}[defi]{Conjecture}

\newtheorem{thr}[defi]{Theorem}

\newtheorem{obs}[defi]{Observation}
\newtheorem{prop}[defi]{Proposition}

\newtheorem{q}[defi]{Question}

\newcommand*{\myproofname}{Proof}
\newenvironment{claimproof}[1][\myproofname]{\begin{proof}[#1]}{\end{proof}}

\def\C{\mathcal{C}}
\DeclareMathOperator{\VC}{VC}
\DeclareMathOperator{\out}{out}
\DeclareMathOperator{\col}{col}
\DeclareMathOperator{\Clique}{Clique}

\DeclareMathOperator{\Star}{Star}

\title{On edge-colouring-games by Erd\H{o}s, and Bensmail and Mc Inerney}
\author{Stijn Cambie 
 \thanks{Department of Computer Science, KU Leuven Campus Kulak-Kortrijk, 8500 Kortrijk, Belgium. Emails: \protect\href{mailto:stijn.cambie@hotmail.com}{\protect\nolinkurl{stijn.cambie@hotmail.com}} and 
\protect\href{mailto:michiel.provoost@kuleuven.be}{\protect\nolinkurl{michiel.provoost@kuleuven.be}}
}
\thanks{Supported by a postdoctoral fellowship by the Research Foundation Flanders (FWO), grant number 1225224N. }
\and Michiel Provoost\footnotemark[1]}

\begin{document}
\parindent=0cm
\maketitle

\begin{abstract}

We study two games proposed by Erd\H{o}s, and one game by Bensmail and Mc Inerney, all sharing a common setup: two players alternately colour edges of a complete graph, or in the biased version, they colour $p$ and $q$ edges respectively on their turns, aiming to maximise a graph parameter determined by their respective induced subgraphs.
In the unbiased case, we give a first reduction towards confirming the conjecture of Bensmail and Mc Inerney, propose a conjecture for Erd\H{o}s' game on maximum degree, and extend the clique and maximum-degree versions to edge-transitive and regular graphs.
In the biased case, the maximum-degree and vertex-capturing games are resolved, and we prove the clique game with $(p,q)=(1,3)$.

%   We study two games proposed by Erd\H{o}s, and one game by Bensmail and Mc Inerney, all with the same setup= of two players alternately colouring an edge of a clique, or the biased version where they colour $p$ resp. $q$ edges of a clique $K_n$ in their turn, trying to maximize a certain graph parameter by the graph induced by their edges.
% In the unbiased case, we provide a first reduction towards confirming the conjecture by Bensmail and Mc Inerney, state a conjecture for Erd\H{o}s' game on maximum degree, and extensions to edge-transitive and, respectively, regular graphs for the clique - and maximum degree version.
%   In the biased case, the maximum degree - and vertex capturing game are resolved, and we prove the clique game with $(p,q)=(1,3)$.

  % We give observations and particular behaviour for each of these problems, and 
  % We state a conjecture for Erd\H{o}s' game on the largest induced maximum degree, and extensions to edge-transitive and, respectively, regular graphs. 
  % We extend winning strategy games where players colour the edges of a clique alternatingly and try to have the best graph induced by their colour, towards the setting of Colex graphs. 
\end{abstract}

\section{Introduction}\label{sec:intro}

% \subsection{Unbiased games}

We consider three games: two proposed by Erdős~\cite{AMM83} (popularised last year at BCC by Cameron~\cite{BCC30} and as a problem in the Erd\H{o}s' problem database, \url{https://www.erdosproblems.com/778}), and one by Bensmail and Mc Inerney~\cite{BM22}, all with a similar setup. We first describe the shared principles of these games.

Consider a graph $G$ of order $n$ and size $m.$
Alice and Bob play a game, taking turns to colour the edges of $G$ red (Alice) and blue (Bob), where Alice is the starting player (player $1$), and Bob is the second player (player $2$).
% Let $A$ be the set of red edges and $B$ the set of blue edges at the end, when all edges have been coloured.
When all edges are coloured, a score, which depends on the set of red edges $A$ and the set of blue edges $B$, is assigned. Here Alice has score $a$ and Bob $b$, where $a$ and $b$ depend on the final colouring of $G$, i.e., on $G[A]$ and $G[B]$ (the graphs induced by the edges in their colour respectively).
In the studied games, $a$ and $b$ will be bounded by the order of the graph, $n$.
Both players play optimally, ensuring that their score is as high as possible relative to the other’s score.
The score of the game, $s(G)$, equals the difference $a-b$.
Alice tries to maximise $s(G)$, while Bob tries to minimise $s(G)$. 
% i.e., both try to win with a gap as large as possible (or if they lose, make the margin as small as possible). 
% Note the difference with Maker-Breaker games, where the players try essentially to maximise/minimise $a$.
To ensure a unique outcome $(a,b)$, we assume Alice tries to maximise $a$ for the optimal value of $s=a-b$, and Bob tries to minimise it for the optimal value of $s=a-b$.
Equivalently, considering the outcome-function $f(G)=(C+1)a-Cb=C(a-b)+a$ for some $C \ge n$, Alice wants $f(G)$ to be as large as possible and Bob wants $f(G)$ to be as small as possible.

By Zermelo's theorem~\cite{Zermelo}, we know $f(G)$ and $s(G)$ are uniquely determined when both players play optimally. So, there is a unique optimal outcome $(a,b).$

The three games we will discuss are called the Clique game, the Star game and the vertex-capturing game.
When played on a graph $G$, the games are denoted by $\Clique(G), \Star(G)$ or $\VC(G)$ respectively, and when $G=K_n$, we write $\Clique(n), \Star(n)$ or $\VC(n).$

\begin{itemize}
  \item In the Clique game, $a = \omega(G[A])$ and $b = \omega(G[B])$.
  % Here the two players try to colour the largest clique in their colour, while limiting the clique size of the other.
  \item In the Star game, $a = \Delta(G[A])$ and $b = \Delta(G[B])$. 
  % Here it is about maximizing the maximum degree or equivalently obtaining the largest star in their colour, while obstructing the opponent of creating a large star in the other colour.
  \item In the vertex-capturing game, at the end of the game; $a$ is the number of vertices for which most incident edges are red, and $b$ is the number of vertices for which most incident edges are blue ($a+b<n$ is possible if there are vertices with half of their incident edges in either colour).
\end{itemize}

The outcome $(a,b)$ and score $a-b$ of the games will be denoted by respectively $\out_{\omega},\out_{\Delta},\out_{\VC}$ and $s_{\omega}, s_{\Delta},s_{\VC}$.% (and $\out_{\col}(G)$ for a later variant) 

By a strategy-stealing argument, one can prove that the first player will never perform worse than the second player.
For details, see the proof of~\cite[Thm.~2.1 \&~2.2]{BM22}, which applies to all three games.
\begin{prop}
  % For each game, Alice's score is at least Bob's, $a \ge b$.
  For every graph $G$, $s_{\omega}(G), s_{\Delta}(G) \in \{0,1\}$ and $s_{\VC}(G) \in \{0,1,2\}$.
  % Furthermore,  for both games $\Clique(G)$ and $\Star(G)$, while $s(G) \in \{0,1,2\}$ for the vertex capturing game.
\end{prop}

For the Clique- and Star-game, we say a graph $G$ is a player 1 win, if $s_{\omega}(G), s_{\Delta}(G) =1$ (the maximum score possible for player $1$), and otherwise it is a player $2$ win.

Biased (asymmetric) variants of a colouring game were defined by~\cite{Kierstead05}, who provided motivation for studying such asymmetric versions.
% In applications where multiple factors have to be taken into account, the problem can become asymmetric and be translated into a winning game (the opponent represents the worst case behaviour).
A problem in certain applications can be translated into an asymmetric game (the opponent represents the worst case behaviour).

In the biased setting of the earlier defined games, the first player (Alice) colours $p$ edges at her turn and the second player (Bob) colours $q$ edges at his turn. Here we assume $q>p$ and Bob wins if at the end $b>a$ for the respective parameters of the game, or $q=p$ and Bob wins if $a=b.$

To specify the parameters $p$ and $q$, we write $\Clique_{(p,q)}(G)$ or $\Clique_{(p,q)}(n)$ and use analogous notation for the other two games.

Erd\H{o}s conjectured that for $\Clique(n)$, the second player can always force a win as long as $n \not =2,$ and asked the same question for $\Star(n)$.

\begin{conj}\label{conj:EP778}[\cite{AMM83}]
$\Clique(n)$ is a player $2$ win, $s_{\omega}(n)=0$, for all $n \in \mathbb N \setminus \{2\}.$
  % $s_{\omega}(n)=0$ for all $n \in \mathbb N \setminus \{2\}.$
\end{conj}

\begin{q}[\cite{AMM83}]\label{q}
  For every $n \in \mathbb N,$ which player wins the game $\Star(n)$ (when is $s_{\Delta}(n)=1$)?
\end{q}

Bensmail and Mc Inerney~\cite[Conj.~7.2]{BM22} conjectured the exact outcome for the vertex-capturing game (on cliques), and confirmed it for $n \le 6$, in~\cite[Prop.~7.6]{BM22}.
\begin{conj}[\cite{BM22}]\label{conj:BM72}
  $$s_{\VC}(n)=\begin{cases}
    2 \mbox{ if } n=2\\
    1 \mbox{ if } n\equiv 3 \pmod 4\\
    0 \mbox{ otherwise. } \\
  \end{cases}$$
\end{conj}

We prove one implication between the validity of~\cref{conj:BM72} for two residue classes modulo $4$.\footnote{This proof by strategy stealing is similar to reductions in~\cite{MS24} and was obtained before their publication, as communicated with them.}

\begin{prop}\label{prop:n1mod4}
  If~\cref{conj:BM72} is true for every $n\equiv 3 \pmod 4,$ then it is true for every $n\equiv 1 \pmod 4$ as well.
\end{prop}

This is done in~\cref{sec:strategystealing}, where, for completeness, we also extend an observation by~\cite{MS24} using strategy stealing, confirming that $\Clique(n)$ being a win for the first player implies that $\Clique(n+3)$ is a win for the second player for every $n$.

Finally, we give our intuition, formulated in conjectures for the Clique and Star game (stated in the framework of edge-transitive and regular graphs). In particular, we are the first to conjecture the answer for~\cref{q}.

\begin{conj}
  With the exception of $K_2$, for every edge-transitive graph $G$, the Clique game is winning for the second player, i.e., $s_{\omega}(G)=0.$
\end{conj}

\begin{conj}
  With the exception of $K_2$ and $K_3$, for every regular graph $G$, the Star game is winning for the second player, i.e., $s_{\Delta}(G)=0.$
\end{conj}

At the end of their conclusion, Malekshahian and Spiro~\cite{MS24} wondered whether studying $\Star_{(p,q)}(n)$ would be interesting. In~\cref{sec:biased_pq}, more precisely~\cref{thm:biased_D_vc}, we notice that both $\Star_{(p,q)}(n)$ and $\VC_{(p,q)}(n)$ are (up to some small cases) player $2$ wins whenever $q>p.$
% The case $p=q$ nevertheless is still challenging. %as even the case $p=q=1$ is not resolved.

In~\cite{MS24}, the authors noted that $\Clique_{(p,q)}(n)$ is a player $2$ win for an exponential bound of $q$ in $p$.
In~\cref{sec:biased_pq}, we improve this drastically to a (unconditional) quadratic bound and a (conditional) linear bound.
Erd\H{o}s~\cite{AMM83} also asked if $\Clique_{(1,2)}{(n)}$ is won by the second player whenever $n\ge 4.$ We conjecture this is indeed the case.
% Using a computer program, it can be checked that the outcome for $4 \le n \le 7$ is equal to respectively $1,1,1,2.$ In this case, Alice colours one edge a time, while Bob can colour two edges in his turn.
% The outcome is $\omega_A-\omega_B$.
Our main result in this direction proves this when $q=2$ is replaced by $q=3$. (Notice that in~\cite{MS24}, it is proved with $q=15$)

\begin{thr}\label{thr:main}
    For every $n\ge 4$, $\Clique_{(1,3)}(n)$ is won by player $2.$
\end{thr}

% The reader may have some intuition, and e.g. Bensmail and Mc Inerney noted that for various graph classed the parity of the order and/or size of the graph determines the result.\\

Despite the interest of many mathematicians, little progress has been made on fully resolving \cref{q} and conjectures~\ref{conj:EP778} and~\ref{conj:BM72}, or the unbiased $p=q$ versions. At the end of our work, we highlight the underlying difficulties that may explain why this is so hard and remains unresolved.

In~\cref{sec:part_beh}, we consider the three games for Colex graphs.
A Colex graph $\C(m)$, where $m=\binom{m_1}{2}+m_2$ with $ 0\le m_2< m_1$, is the graph formed by adding $m_2$ edges between $K_{m_1}$ and an additional vertex (if $m_2>0$). See~\cref{fig:C(12)} for examples where $(m_1,m_2) \in\{(4,3), (5,2)\}.$
As a class of graphs that smoothly extend cliques, this serves as a showcase to challenge certain intuitive expectations by exhibiting particular behaviours.
The results were found by computer code (which implements the idea of Zermelo's theorem~\cite{Zermelo}), explained in~\cref{sec:pc_code}.

\section{Further results by strategy stealing}\label{sec:strategystealing}

In this section, we present the proof of~\cref{prop:n1mod4}.

\begin{proof}[Proof of~\cref{prop:n1mod4}.]
  Let $n \equiv 1 \pmod 4$ where $n\ge 5$.
  Alice colours one edge, which we call $uv$, red.
  Bob plays on the other vertices, which induce a $K:=K_{n-2}$, optimally (being the first player there).
  Whenever Alice colours an edge $ux$ or $vx$, for $x \in V(K),$ Bob colours the other of the two in his colour.
  If Alice colours an edge in $K,$ so does Bob.
  Since $s_{\VC}(K)=1,$ and for every $x \in V(K),$ there are two edges outside $K$ with different colours, Bob captures one vertex more among those $n-2$ (by the assumption that ~\cref{conj:BM72} is true for $n-2\equiv 3 \pmod 4$).
  By the pigeon hole principle, one vertex of $\{u,v\}$ has at least $\ceilfrac{n-2}{2}=\frac{n-1}{2}$ incident edges coloured by Bob.
  Hence Alice will have captured exactly one vertex among $u$ and $v.$
  This implies that Bob does not lose.
  By a strategy stealing argument, Alice will not lose either (\cite[Thm.~2.2]{BM22}), and this implies that $s_{\VC}(K_n)=0.$
\end{proof}

As a corollary of~\cref{prop:n1mod4}, together with the computer verification for $n \le 8$, we know that~\cref{conj:BM72} is true for $n \le 9.$

When $4 \mid n,$ the truth of~\cref{conj:BM72} appears natural since we know that $s_{\VC}(G) \in \{0,2\}$, while the size and order are both even. Nevertheless, taking into account that $\VC( \C(12))\not=0$ (see~\cref{tab:Results_games_onColexGraphs}) while $\C(12)$ has even order and even size, a proof cannot be trivial.
Also one might think about mirror strategies creating isomorphic coloured graphs on even size cliques $K_n.$ Nevertheless, already for $n \in \{4,5\}$ it can be checked that there is no winning strategy for the second player for the creation of isomorphic coloured graphs.

In~\cite{MS24}, the authors proved that if Alice wins $\Clique(n)$ (something that conjecturally is only the case when $n=2$), then Bob wins $\Clique(n+3)$ for $n \ge R(5,5).$
This was later included on~\url{https://www.erdosproblems.com/778} (at that time) without the condition that $n$ has to be sufficiently large.
Here we prove that it is indeed a valid statement for all $n$. This follows in essence from~\cite[Rem.~8]{MS24} with the additional observation that $n \ge R(3,4)=9$ instead of $n\ge R(4,4)=18$ also works, and the computer check for $n \le 8$ (checked with independent implementation for $n \le 7$).

\begin{prop}
  If Alice wins $\Clique(n)$ for $n \ge 8$, then Bob wins $\Clique(n+3).$
  Bob wins $\Clique(n)$ for $3 \le n \le 8.$
\end{prop}
\begin{proof}
  The fact that Bob wins $\Clique(n)$ for $3 \le n \le 8$ was done by the implementation. 

  Next, we observe that if Alice wins $\Clique(n)$ for $n \ge 8$, she has at least $4$ points, $\omega(G[A]) \ge 4$.

\begin{obs}
  Every graph with order $n = 8$ and size $14$ satisfies $\omega(G) \ge 3 $ or $\alpha(G)\ge 4$. 
\end{obs}

\begin{claimproof}
  Easily verified by exhaustive computer check. 
  % The statement for $n=8$ also follows from the $n=7$ case. 
  The three graphs of order $8$ satisfying $\omega(G) =2$ and $\alpha(G)=3$ have size $10, 11$ and $12$ (and $R(2,4)<R(3,3)<8$). For more info, one can find these graphs by their HoG Id $588, 640, 26996$ at \url{https://houseofgraphs.org} (\cite{HOG}).
\end{claimproof}

Now, we show a strategy of Bob to ensure he wins $\Clique(n+3).$
At his first move, he colours an edge adjacent with the edge chosen by Alice. Next, Alice can choose between $5$ configurations up to isomorphism (there are $5$ edge orbits).

\begin{figure}[h]
  \centering
  \begin{tikzpicture}[scale=0.6]
    
    \draw[thick, dashed, red] (0:2)--(120:2)--(240:2);
    \draw[thick, blue] (0:2)--(240:2);
    \foreach \t in {0,120,240}{
    \draw[fill] (\t:2) circle (0.25);
    }
  \end{tikzpicture}\quad
  \begin{tikzpicture}
    
    \draw[thick, dashed, red] (0,2)--(0,0);
    \draw[thick, dashed, red] (2,2)--(2,0);
    \draw[thick, blue] (0,0)--(2,0);
    \foreach \t in {0,2}{
    \foreach \y in {0,2}{
    \draw[fill] (\t,\y) circle (0.15);
    }
    }
  \end{tikzpicture}\quad
  \begin{tikzpicture}
    
    \draw[thick, dashed, red] (0,2)--(0,0)--(2,0);
    \draw[thick, blue] (2,2)--(2,0);
    \foreach \t in {0,2}{
    \foreach \y in {0,2}{
    \draw[fill] (\t,\y) circle (0.15);
    }
    }
  \end{tikzpicture}\quad
  \begin{tikzpicture}
    \draw[thick, dashed, red] (0,2)--(0,0)--(2,2);
    \draw[thick, blue] (0,0)--(2,0);
    \foreach \t in {0,2}{
    \foreach \y in {0,2}{
    \draw[fill] (\t,\y) circle (0.15);
    }
    }
  \end{tikzpicture}\qquad
  \begin{tikzpicture}[scale=0.6]
    \draw[thick, dashed, red] (-2,-1.732)--(-2,1.732);
    \draw[thick, dashed, red] (120:2)--(240:2);
    \draw[thick, blue] (0:2)--(240:2);
    \foreach \t in {0,120,240}{
    \draw[fill] (\t:2) circle (0.25);
    }
    \foreach \y in {-1.732,1.732}{
    \draw[fill] (-2,\y) circle (0.25);
    }
  \end{tikzpicture}
  \caption{The $5$ possible situations after $3$ edges have been coloured.}
  \label{fig:3edgescoloured}
\end{figure}

Bob colours a second edge, not incident with his first, to obtain one of the following three configurations (\cref{fig:4edgescoloured}). Note that his lower (the lower one drawn in \cref{fig:4edgescoloured}) blue edge is incident with both of the red edges. 
One can check each of those $3$ cases separately. Independently of what Alice does in her third turn, Bob can answer, ensuring a triangle $uv_1v_2$ with two out of three edges coloured blue (called $uv_1$ and $uv_2$), the subgraph spanned by $V \setminus \{u, v_1, v_2\}$ containing exactly one coloured edge, which is blue. We show this by giving his action in three cases.

\begin{itemize}
    \item 
If Alice's third coloured edge is not incident with any of the previous coloured edges, Bob can colour an edge incident to that edge and his lower edge.

\item
If Alice chooses an edge which contains one additional vertex, Bob can colour an edge connecting that additional vertex with his lower edge.
\item
If Alice chooses an edge among vertices incident to other coloured edges, either Bob can colour an edge incident to a new vertex and his lower edge, or between the lower edge and the unique vertex which was incident to red edges, but none of the blue edges.
\end{itemize}

In the remaining of the game, Bob colours an edge spanned by $V \setminus \{u, v_1, v_2\}$ whenever Alice coloured one there, using the strategy on the $K_n$ spanned by $V \setminus \{u, v_1, v_2\}$ to win (as first player).
If Alice coloured an other edge (incident with at least one of $\{u,v_1,v_2\})$, Bob colours an uncoloured edge from a pair $\{v_1x, v_2x\}$ with $x \in V \setminus \{u, v_1, v_2\}$ where the other edge is red, whenever possible, and otherwise a random edge.
By strategy stealing, Bob can obtain a larger clique on the $K_n$ spanned by $V \setminus \{u, v_1, v_2\}$ and thus not lose on $K_{n+3}$. For this, note that the largest clique on $K_{n+3}$ coloured by Alice is bounded by the largest one containing $v_1v_2$ (at most a $K_4$, since she could not colour three additional pairs in red) and the largest clique on $V \setminus \{u, v_1, v_2\}$ plus one.

\begin{figure}[h]
  \centering
 
  \begin{tikzpicture}
    \draw[thick, dashed, red] (1,2)--(0,0);
    \draw[thick, dashed, red] (1,2)--(2,0);
    \draw[thick, blue] (0,0)--(2,0);
    \draw[thick, blue] (1,2)--(3,2);
    \foreach \t in {0,2}{
    \foreach \y in {0,2}{
    \draw[fill] (\t+\y/2,\y) circle (0.15);
    }
    }
  \end{tikzpicture}\quad
  \begin{tikzpicture}
    
    \draw[thick, dashed, red] (0,2)--(0,0);
    \draw[thick, dashed, red] (2,2)--(2,0);
    \draw[thick, blue] (0,0)--(2,0);
    \draw[thick, blue] (4,2)--(2,2);
    \foreach \t in {0,2}{
    \foreach \y in {0,2}{
    \draw[fill] (\t,\y) circle (0.15);
    }
    }
    \draw[fill] (4,2) circle (0.15);
  \end{tikzpicture}\quad
  \begin{tikzpicture}
    \draw[thick, blue] (4,2)--(2,2);
    \draw[thick, dashed, red] (0,2)--(0,0)--(2,2);
    \draw[thick, blue] (0,0)--(2,0);
    \foreach \t in {0,2}{
    \foreach \y in {0,2}{
    \draw[fill] (\t,\y) circle (0.15);
    }
    }
    \draw[fill] (4,2) circle (0.15);
  \end{tikzpicture}
  \caption{The $3$ possible situations after Bob colours its second edge.}
  \label{fig:4edgescoloured}
\end{figure}
\end{proof}

\section{The biased games for large $n$}\label{sec:biased_pq}% $\Delta$ and vertex-capturing game

We answer the final question in the conclusion of~\cite{MS24}, by noticing the relation with a result on the (degree) balancing game by Beck~\cite[Thm.~17.5]{Beck08}.

\begin{thr}\label{thm:biased_D_vc}
    When $p<q$ are integers, then $\Delta_{(p,q)}(n)$ and $\VC_{(p,q)}(n)$ are both player $2$ wins provided that $n$ is sufficiently large.
\end{thr}

\begin{proof}
    Bob can play as the first player for the balancing game from~\cite[Thm.~17.5]{Beck08}, neglecting the first $p$ edges from Alice.
    % As such, he can enforce that all the degrees in his colour have degree $\frac{q+o_n(1)}{p+q}n$ (note that $p=o(n)$ as well).
Consequently, he can ensure that every vertex has blue degree $\frac{q+o_n(1)}{p+q}n$ (recall that $p=o(n)$ as well). For $n$ sufficiently large, all of these degrees are larger than $\frac n2$ and so Bob wins both the maximum degree game and the vertex capturing game, since $\delta_B> \frac n2 > \Delta_A$.
\end{proof}

It can be observed that for $(p,q)=(1,2),$ $n \ge 3$ is sufficient.

\begin{prop}\label{prop:(1,2)}
    The games $\Delta_{(1,2)}(n)$ and $\VC_{(1,2)}(n)$ are both player $2$ wins provided that $n\ge 3.$
\end{prop}

\begin{proof}
    Write $n=3k+r$, where $k \ge 1$ and $0 \le r \le 2.$
    Let the vertices be partitioned in $U=\{u_1, u_2, \ldots, u_k\}, V=\{v_1,v_2, \ldots, v_k\}$ and $W=\{w_1, \ldots, w_k\}$, and also $x$ if $r=1$ or $x$ and $y$ if $r=2$.

    Bob colours the two other vertices for triples of the following form:
    $$\{u_iu_j,v_iv_j,w_iw_j\}, \{u_iv_i, u_iw_i,v_iw_i\}, \{u_iv_j,v_iw_j,w_iv_j\},\{xu_i,xv_i,xw_i\},\{yu_i,yv_i,yw_i\},$$ where $1 \le i,j \le k.$
    In this case the number of blue incident edges to (at least one) of $u_i, v_i, w_i$ is double the number of red such edges.
    As such, $\Delta_A< \Delta_B$ and Bob captured at least $k$ vertices more than Alice did.
\end{proof}

Next, we consider regimes in which $\Clique_{(p,q)}(n)$ is won by the second player.

By applying~\cite[Thm.~2]{ES72}, we get a general improvement for the general biased clique game.
We recall the latter statement for convenience, formulated slightly different (the way applied in~\cite[Thm.~1]{ES72}).

\begin{thr}[\cite{ES72}]\label{thr:ES72}
    Let $G$ be a graph of order $n$ and size bounded by $\frac{\binom n2}{k}$. Then $\max\{\omega (G), \alpha (G)\} > c_4 \frac k{\log k} \log n.$
\end{thr}

\begin{thr}\label{thr:game1_biased}
    There is a constant $c$ such that if $q>cp^2$, then for every $n$ sufficiently large, $\Clique_{(p,q)}(n)$ is won by player $2.$
\end{thr}

\begin{proof}
    Let $c=\frac 2{c_4},$ where $c_4$ is as in~\cref{thr:ES72}.
    Then for $n$ sufficiently large, the ratio of edges coloured by the first player is bounded by $p+\frac{p}{p+q}\binom{n}{2}<\frac{\binom n2}{k}$ for $q+1 \ge k\ge \frac qp > \frac 2{c_4}p.$
    By~\cref{thr:ES72}, the graph of at least one of the two players contains a clique of order no less than $c_4 \frac k{\log k} \log n.$
    
    By~\cite[Thm.~20.1]{Beck08} on the other hand, the second player can force the largest clique induced by the colour of the first player to be bounded by $ \frac{2p}{\log(q+1)} \log n.$

    Since $c_4 \frac k{\log k}>\frac{2p}{\log(q+1)},$ we conclude the second player wins.
\end{proof}

For further (conditional) improvement, we consider Maker-Breaker games, in which Maker tries to obtain a clique $K_q$ in its colour, and Breaker tries to prevent that.

Let \( f_N(m, b) \) denote the largest \( q \) such that Maker can occupy a \( K_q \) in the \( (m : b) \) game (Maker has $m$ edges and Breaker has $b$ edges) played on $K_N$.
The value of $f_N(1,1)\sim 2 \log_2 N$ is known exactly by Beck~\cite[Thm.~6.4]{Beck08}. Beck~\cite[Thm.~30.2]{Beck08} also gave lower bounds for $f_N(m,1)$ of the form $2(1+o(1))\log_{(m+1)/m}(N)$ and
Gebauer~\cite[Thm.~1.3]{Gebauer12} gave a lower bound on \( f_N(m, b) \), $(1+o(1))m \log_{b+1}(N)$.

Motivated by these lower bounds, we state the following conjecture.

\begin{conj}\label{conj:strictmonotonicity}
    For fixed $b,m$, for $N$ sufficiently large $f_N(m, b+1)<f_N(m, b)<f_N(m+1, b)$ and the differences tend to infinity as $N$ goes to infinity.
\end{conj}

The first inequality for $m=1$ would follow from the unsolved final conjecture by Bednarska and {\L}uczak~\cite{BL00}.

This strict monotonicity immediately implies that~\cref{thr:game1_biased} can be further improved.

\begin{thr}
    If~\cref{conj:strictmonotonicity} is true, then for fixed $p$ and $n$ sufficiently large, $\Clique_{(p,2p+1)}(n)$ is won by player $2.$
\end{thr}

\begin{proof}
    The second player uses $p$ edges as Breaker for the formation of a clique by the first player, and $p+1$ edges to form a large clique itself.
    The first $p$ edges of player $1$ can be omitted in the building process of the second player, and if his own edges are already placed at optimal places, other edges can be put randomly.
    We deduce that $\omega_A \le f_N(p,p)$ and $\omega_B \ge f_N(p+1,p)-p.$
    By assumption, for $N$ sufficiently large this implies $\omega_B>\omega_A$.
\end{proof}

By the results from Beck, this in particular implies that the second player wins $\Clique_{(1,3)}(n)$ provided that $n$ is sufficiently large. In~\cref{sec:biasedclique}, we prove this for all $n\ge 4$.

It is also natural to conjecture the following sharp version for the game.

\begin{conj}
    Provided that $n$ is sufficiently large, the $\Clique_{(p,p+1)}(n)$ is won by player $2.$
\end{conj}

\section{The biased clique game $\omega_{1,3}$}\label{sec:biasedclique}

% We remark that some slight variants of the main theorem are much easier.
% For example if Bob is also allowed to play first.

There are multiple elementary proofs possible when $n$ is even. We provide the one solution we could modify for the odd case.

\begin{thr}\label{thr:game2_n_even}
    For every $n\ge 4$ even, $\Clique_{(1,3)}(n)$ is won by player $2.$
\end{thr}

\begin{proof}
    We will consider a turn as one move (colouring an edge) by Alice and (at most) $3$ moves of Bob.
    Since Bob can also play some of the remaining edges randomly, the win for Bob will be kept, so we present the strategy for Bob in this simpler way.

    The $n=2k$ vertices are initially unlabeled vertices.
    During the game, a vertex is labeled as soon as at least one of its incident edges is coloured, or when there are only two vertices without coloured incident edges, in which case both are labeled.

    \textbf{Initial play}
    Here we assume that at the beginning of the turn, at least $4$ vertices are unlabeled.
    The $2m$ labeled vertices will have indices which appear in pairs $(v_i, v_{i+k})$, where $1\le i \le m$.
    The edges $v_i v_{i+k}$, will be coloured by Bob.
    If an edge $v_iv_j$ is coloured by Alice, the symmetric edge $v_{i+k}v_{j+k}$ will have been coloured by Bob.
    Here indices can be considered modulo $2k$ throughout.

    Depending on the edge Alice chooses, Bob reacts as follows.
    
    \begin{itemize}
        \item If Alice colours an edge between two unlabeled vertices, we will label them as $v_i, v_j$ (where $i=m+1, j=m+2$).
    Bob picks two unlabeled vertices, calls them $v_{i+k},v_{j+k}$,
    and colours the edges $v_{i}v_{i+k}, v_{i+k}v_{j+k},v_{j}v_{j+k}$.

        \item If Alice colours an edge whose endvertices contain a labeled one, $v_i$, and one unlabeled one, which we call $v_{m+1}.$ Bob takes an other unlabeled vertex, $v_{m+1+k},$ and colours
    $v_{i+k}v_{m+1+k}$ and $v_{m+1}v_{m+1+k}$.
        \item If Alice colours an edge $v_i v_j$ whose end vertices are already labeled, Bob colours $v_{i}v_{j+k}, v_{i+k}v_{j+k}$ and $v_{j}v_{i+k}$ (some of these edges may have been coloured before by Bob if the vertices $v_i, v_j$ were labeled in the same turn).
    \end{itemize}

    If two unlabeled vertices remain at the end of the turn, we label them by $x$ and $y$, and we go to the final play.

    \textbf{Finishing play}

\begin{itemize}
        \item If Alice colours $xy$, Bob colours no edges.

        \item If Alice colours $xv_i$, then Bob colours $xv_{i+k},yv_i$ and $yv_{i+k}$.

        \item If Alice colours $yv_i$, then Bob colours $xv_i,xv_{i+k}$ and $yv_{i+k}$.

        \item If Alice colours an edge $v_i v_j$ whose end vertices are already labeled, then Bob responds as during the initial play; that is, he colours $v_{i}v_{j+k}, v_{i+k}v_{j+k}$ and $v_{j}v_{i+k}$ (some of these edges may have been coloured before by Bob if the vertices $v_i, v_j$ were labeled in the same turn).
\end{itemize}

Having sketched a winning strategy for Bob, we prove that Bob wins.

Consider a largest clique induced by all vertices different from $x$ and $y$ for which all edges are coloured by Alice.
Up to shifting and relabeling opposite vertices,
we denote it $v_1v_2 \ldots v_r.$ This clique has order $r \ge 2.$
In that case, by the strategy of Bob, he coloured the clique $v_{k+1}v_{k+2} \ldots v_{k+r}$ containing all symmetric edges.

There are at most $r-1$ edges among $G[v_1,v_2, \ldots ,v_r]$ which have been coloured by Alice at a moment at least one of its endvertices was unlabeled.
For each of these edges $v_iv_j$, the clique $G[v_i,v_j,v_{i+k},v_{j+k}] \cong K_4$ contains two edges coloured by Alice.
If $v_iv_j$ was already labeled before Alice coloured that edge, it is the only edge within $G[v_i,v_j,v_{i+k},v_{j+k}]$ which is coloured by her.
This implies that the bipartite graph spanned by the partition classes $V_1=\{v_1,v_2, \ldots ,v_r\}$ and $V_2=\{v_{k+1},v_{k+2}, \ldots ,v_{k+r}\}$ has at most $r-1$ blue edges.
This implies that there is a vertex $v_i$ in $V_1$ such that $G[v_i,V_2]$ is fully coloured by Bob. That is, $G[V_2 \cup \{v_i\}]$ is a clique of order $r+1$ fully coloured by Bob.

By Bob's strategy, Alice has no triangle containing the edge $xy$ that is fully coloured.
If Alice has a large fully coloured clique containing $x$, then it must be of the form $V_1 \cup \{x\}$; 
analogously, if it contains $y$, it must be of the form $V_1 \cup \{y\}$.
But then the clique spanned by $V_2 \cup \{v_i, y\}$ (resp. $V_2 \cup \{v_i, x\}$) is fully coloured by Bob. 
\end{proof}

\begin{thr}\label{thr:game2_n_odd}
    For every $n\ge 5$ be odd, $\Clique_{(1,3)}(n)$ is won by player $2.$
\end{thr}

\begin{proof}
    The setup and initial play proceed as in~\cref{thr:game2_n_even}, up to the point where either three unlabeled vertices remain and Alice colours an edge incident to at least one of them, or five unlabeled vertices remain and Alice begins colouring an edge between two of these.
    The labeled vertices (possibly after relabeling), can be $L=\{v_i \mid 1 \le i \le 2k\}$ where $v_i$ and $v_{i+k}$ form a pair of opposite vertices. Here $n=2k+3$ or $n=2k+5$.

    When $n=2k+3$, we call the the three remaining unlabeled vertices $a,b,c.$
    
    When $n=2k+5$, we name the end vertices of the edge coloured by Alice $x,y$. We label the other three by $a,b,c.$
    
    We present the final play in each of the following three situations.

    \underline{Case 1: $n=2k+3$ and Alice colours $av_\ell$}

    Bob colours the edges of $abc$.

    From now on;

    \begin{itemize}
    \item If Alice colours an edge in $\{a,b,c\} \times \{v_\ell, v_{\ell +k}  \}$, then Bob can colour edges among $\{a,b,c\} \times \{v_\ell, v_{\ell +k}\}$ (but does not even have to).
    \item If Alice colours an edge $v_i v_j \in G[L]$, Bob colours $v_{i}v_{j+k}, v_{i+k}v_{j+k}$ and $v_{j}v_{i+k}$ (some of these edges may have been coloured before by Bob if the vertices $v_i, v_j$ were labeled in the same turn).
    \item If Alice colours an edge in $dv_i$ for some $d \in \{a,b,c\}$, then Bob colours $av_{i+k}, bv_{i+k}$ and $cv_{i+k}$. If those three were already coloured, Bob colours the uncoloured edge among $av_{i}, bv_{i}$ and $cv_{i}$.
    \end{itemize}

    If a largest clique of Alice is contained in $G[L],$ Bob has a larger clique within $G[L],$ as shown before in the proof of~\cref{thr:game2_n_even}.

    If the largest clique contains one vertex from $\{a,b,c\}$ and w.l.o.g. $V_1=\{v_1,v_2, \ldots ,v_r\}$ ($r$ vertices from $L$.

    Let $V_2=\{v_{k+1},v_{k+2}, \ldots ,v_{k+r}\}$.
    
    Bob has a clique spanned by the vertices $V_2 \setminus \{v_\ell, v_{\ell +k}\} \cup \{a,b,c\}$, so its clique number is at least $r+2.$

    \underline{Case 2: $n=2k+3$ and Alice colours $ab$}
    
    Bob reacts by colouring the other two edges of $abc$.
    
    From now on;

    \begin{itemize}
    \item If Alice colours an edge $av_i$, then Bob colours
    $bv_i, bv_{i+k}$ and $cv_{i+k}.$
    Symmetrically, if Alice colours an edge $bv_i$, then Bob colours $av_i, av_{i+k}$ and $ cv_{i+k}.$
    \item If Alice colours an edge in $cv_i$, then Bob colours
    $ av_{i+k}, bv_{i+k}$ and $ cv_{i+k}.$
    \item If Alice colours an edge $v_i v_j \in G[L]$, Bob colours $v_{i}v_{j+k}, v_{i+k}v_{j+k}$ and $v_{j}v_{i+k}$ (some of these edges may have been coloured before by Bob if the vertices $v_i, v_j$ were labeled in the same turn).
    \end{itemize}

If a largest clique of Alice is contained in $G[L],$ Bob wins as before.

If a largest clique of Alice is of the form $\{a,v_1,v_2, \ldots ,v_r\}$, then Bob has a clique spanned by $\{b,c,v_{k+1},v_{k+2}, \ldots ,v_{k+r}\}$.
Similarly, $\{b,v_1,v_2, \ldots ,v_r\}$ or $\{c,v_1,v_2, \ldots ,v_r\}$ would be a smaller clique than $\{a,c,v_{k+1},v_{k+2}, \ldots ,v_{k+r}\}$.

    \underline{Case 3: $n=2k+5$ and Alice colours $xy$}

    Bob reacts by colouring the edges of $abc$.
    
    From now on;

    \begin{itemize}
    \item If Alice colours an edge in $av_i,bv_i$ or $cv_i$, then Bob colours
    $ av_{i+k}, bv_{i+k}$ and $ cv_{i+k}.$
    \item If Alice colours $xv_i$ or $yv_i,$ Bob colours the three remaining edges in $\{x,y\} \times \{v_i,v_{i+k}\}.$
    \item If Alice colours w.l.o.g. $ax$, Bob colours $ay,by$ and $cy.$
    (the $6$ edges are all equivalent)
    If later, Alice colours $bx$ or $cx$, Bob colours the other edge among those two. 
    \item If Alice colours an edge $v_i v_j \in G[L]$, Bob colours $v_{i}v_{j+k}, v_{i+k}v_{j+k}$ and $v_{j}v_{i+k}$ (some of these edges may have been coloured before by Bob if the vertices $v_i, v_j$ were labeled in the same turn).
    \end{itemize}

    If a maximum clique from Alice contains at most one vertex from $\{a,b,c,x,y\},$ then Bob has a larger clique as before.

    The only clique of Alice containing $xy$ is $xy$ itself.

    If $\{a,x,v_1,v_2, \ldots ,v_r\}$ is a maximum clique from Alice, then $\{b,c,y,v_{k+1},v_{k+2}, \ldots ,v_{k+r}\}$ is a monochromatic clique coloured by Bob.
    In all cases, the maximum clique coloured by Bob is strictly larger than the maximum clique coloured by Alice.    
\end{proof}

\section{Particular behaviour for some variants and extensions
% Particular behaviour for the extension to Colex graphs
}\label{sec:part_beh}

\subsection{Smoothening cliques; the games on Colex graphs}\label{subsec:ColexGraphs}

Playing the Clique game on $K_{n+1}$ instead of $K_n$ implies that there are many more edges to colour.
To have a smoother behaviour, we can consider Colex graphs. 
Let $n$ be the smallest integer for which $\binom n2 \ge m.$
The Colex graphs $G=\C(m)$ of size $m$ can be defined (up to isomorphism) by choosing $V(G)=\{1,2,\ldots, n\}$ and letting $E(G)$ consist of the first $m$ pairs of vertices in colexicographic order (here $(a,b)<(c,d)$ if $b<d$ or $a<c$ and $b=d$). Equivalently, it is the clique $K_n$ with the edges of a star of size $\binom n2 - m$ removed.

\begin{figure}[h]
  \centering
  \begin{tikzpicture}
    \foreach \t in {72,144,216,288,360}{
    \draw (\t:3)--(\t+144:3);
    }
    \foreach \t in {72,216,288,360}{
    \draw (\t:3)--(\t+72:3);
    }
    \draw[line width=0.75mm, blue] (-72:3)--(72:3)--(0:3)--cycle;
    % \draw[thick, red] (-144:3)--(144:3);
    \foreach \t in {72,144,216,288,360}{
    \draw[fill] (\t:3) circle (0.15);
    }
  \end{tikzpicture}\quad
  \begin{tikzpicture}
    \foreach \t in {72,144,216,288,360}{
    \draw (\t:3)--(\t+144:3);
    \draw (\t:3)--(\t+72:3);
    \draw[fill] (\t:3) circle (0.15);
    }
    \draw[line width=0.75mm, blue] (-144:3)--(144:3);
    \draw (144:3)--(-4.5,0) ;
    \draw (-144:3)--(-4.5,0) ;
     \draw[fill] (-4.5,0) circle (0.15);
     \draw[fill] (144:3) circle (0.15);
     \draw[fill] (-144:3) circle (0.15);
  \end{tikzpicture}
  \caption{The Colex graphs $\C(9)$ and $\C(12)$.}
  \label{fig:C(12)}
\end{figure}

A game can be played on any Colex graph $\C(m).$ If the addition of one edge, when going from $\C(m)$ to $\C(m+1),$ would have a small impact in the sense that the outcome $(a,b)$ differs in at most one entry, then by Ramsey theory (or noting the outcomes are not bounded) there would be infinitely many graphs $\C(m)$ for which the first player can win.

To have even further information, we consider also the Colex variant of the Clique game. Here the players try to form a Colex graph of larger size. 
Observe that when there is a tie in the Colex game, there is also a tie in the Clique game (and even the exact outcome can be deduced).
We verified the outcome of the games for some Colex graphs, which are summarised in~\cref{tab:Results_games_onColexGraphs}, and observed that each game exhibits some surprising behaviour.
Analogous to the other games, the outcome of the Colex game is denoted with $\out_{\col}(G)$.

\begin{table}[h]
  \centering
  \begin{tabular}{|c|c|c|c|cc|}
  \hline
  m & $G=\C(m)$ & $\out_{\col}(G)$ & $\out_{\Delta}(G)$& $s_{\VC}(G)$&  $\out_{\VC}(G)$\\
  \hline 
  1 & $K_2$ & \textit{(1,0)} & \textit{(1,0)} & \textit{2} & (2,0) \\
  2 & $P_3$ & (1,1) & (1,1) & 0  &(1,1) \\
  3 & $K_3$& \textit{(2,1)} & \textit{(2,1)} & 1 & (1,0) \\
4 & & (2,2) & (2,2) & 1 & (2,1)\\
5 & & \textit{(2,1)} & \textit{(2,1)} & \textit{2}  &(2,0) \\
 6 & $K_4$& (2,2) & (2,2) & 0  &(2,2) \\
7 &  & (2,2) & \textit{(3,2)} & 1  &(3,2) \\
8 & & \textcolor{red}{(2,2)} & \textbf{\textcolor{green}{\textit{(3,2)}}} & 1  &(2,1)\\
9 & $K_5^-$ & \textcolor{red}{\textit{(4,2)}} & \textit{(3,2)} & 1 & (2,1)\\
 10 & $K_5$& \textcolor{red}{(4,4)} & (3,3) & 0 & (1,1) \\
11 & & (4,4) & (3,3) & \textit{2} & (3,1)\\
12 & & (4,4) & (3,3) & \textbf{\textcolor{red}{1}} & (2,1)\\
13 & & (4,4) & (3,3) & 1 & (3,2)\\
14 & $K_6^-$ & \textbf{\textcolor{orange}{(4,4)}}& (3,3) & 0 & (3,3)\\
 15 & $K_6$& \textbf{\textcolor{orange}{(5,5)}} & (4,4) & 0 & (3,3)\\
 16 & & (5,5) & (4,4) & 1 & (4,3)\\
17 & & (5,5) & \textbf{\textcolor{red}{\textit{(4,3)}}} & 1 & (3,2)\\
18 & & (5,5) & (4,4) & 0 & (3,3) \\
19 & & (5,5) & (4,4)  & 1 & (3,2) \\
20 & & (5,5) & (4,4)  & 0 & (2,2) \\
21 & $K_7$& (5,5) &  (4,4) & 1 & (2,1)\\
22 & & (5,5) & (4,4) & 1 & (2,1) \\
23 & &(5,5) & (4,4) & 1 & (3,2) \\
24 & & (5,5) & (4,4)& 0 & (3,3) \\
% 25 & &(5,5)  & & \\
% 26 & & \underline{(5,5)} & & \\
% 27 & & \underline{(5,5)} & & \\
28 & $K_8$& (5,5) & (5,5) & 0 & (4,4)\\

\hline
\end{tabular}
  \caption{Outcomes of the three symmetric games on Colex graphs with small size.}
  \label{tab:Results_games_onColexGraphs}
\end{table}

\begin{itemize}
  \item For the Colex game, the increase of outcomes can occur in at least two ways. Increasing the size from $\C(8)$ to $\C(10),$ implies that initially (for $m=9$) the first player can increase its score $a$, but next (at $m=10$) also the second player can do so. Note that $\C(9)$ can be considered as two $K_4$s overlapping in a triangle. The first player, Alice, can colour two of the three edges within this triangle and can use this advantage to create a $K_3$ in her colour.
  On the other hand, the extra edge going from $\C(14)$ to $\C(15)$ implies that both players can increase their number of points. The extra edge implies that the playing field is enlarged for both players.
  This behaviour is expected to be typical for transitions in large instances.
  \item For the Star (maximum degree) game, we observe that Alice wins on $\C(7), \C(8)$ and $\C(9)$. The outcome changes for $16 \le m \le 18.$ When $m=\binom{n-1}{2}+2,$ there is an edge that is the only one in its orbit (e.g. the thick blue one in $\C(12)$ in~\cref{fig:C(12)}).
  \item In the vertex capturing game, the first player can win on $\C(12)$ despite the order and the size of $\C(12)$ being even. 
  % While~\cref{conj:BM72} gives some periodicity modulo $4$, we note that the outcome of $VC(\C(m)}$ for $m=\binom{n-1}{2}+r$ and $m=\binom{n+3}{2}+r$ can be different.
\end{itemize}

The behaviour sketched here also suggests that one cannot easily prove that $s_{\omega}(n)=1$ holds in only $o(x)$ many cases for $3 \le n \le x$.

\subsection{The Clique game in less symmetric situations}\label{subsec:CliqueGameOnIrregular}

The question if irregular graphs behave different than cliques was asked in \cite[Ques.~1]{MS24}.

\begin{q}\label{ques:increasing_omega}
  Do there exist graphs $H$ with arbitrarily large clique number such that
  $s_{\omega}(H)=1$?
\end{q}

The main difference between irregular graphs and cliques is that some edges are more important to claim than others in the start, as was the case for $K_5^-.$
The game played on the graph join $K_3 +2tK_1$ ($K_{3,2t}$ with $3$ edges added between the vertices of degree $2t$) can result in many (t) triangles against none. The latter is true because there are $3$ edges belonging to $2t+1$ triangles and the others to only two, so having the advantage to pick two out of three edges with larger impact ($K_3$-multiplicity) makes the difference.

If there exists a graph $H$ such that Alice wins the clique game on $H$ whenever a red $K_3$ is already present before her move, with outcome $(a,a-1)$ where $a\ge 4,$ then there would be a chance to answer~\cref{ques:increasing_omega} positively for outcome $(a,a-1)$ as well.
For this, one could extend $K_5^-$ by attaching copies $H_i$ of the graph $H$, such that $H_i \cap K_5^- \cong K_3 $ for every $i$, and $V(H_i) \cap V(H_j)\subset V(K_5^-)$ for $i \not=j.$

Nevertheless, if Alice plays the clique game on a $K_8$, where she starts the game after an advantage of a $K_3$ in her colour is already given, the game still ends with $K_3$ against $K_3$ (outcome $(3,3)$). Starting with a claw, we even observe the following non-smooth jump.

\begin{obs}
  If Alice starts with a star (claw $K_{1,3}$) $S_4$ in her colour as advantage (here the claw spans only maximum degree vertices of the Colex graph) before colouring her first chosen edge, the outcome of the modified Clique game is $(4,4)$ on $K_8$ and $(3,3)$ when played on $\C(26)$ or $\C(27).$
\end{obs}

In particular, related to~\cref{ques:increasing_omega}, we cannot exclude even that for every graph $G$ satisfying $\omega(G) \ge 5$, $\Clique(G)$ is a player $2$ win.

\section{Details about computer program}\label{sec:pc_code}

In \cite{github-repository}, we included two independent implementations for the computer-aided verifications, one in C and one in Sage. The Sage implementation has simpler syntax, and can be found in the files with name \textit{EP\&BMI\_Games}.
This section details the strategies used in the C program(s) to solve the games.  
The publicly available code also serves as a tool for other researchers. 
The code is publicly available and may serve as a tool for further investigation.
% We encourage interested readers to use this code as a tool to gain intuition in other variants of the game or the behaviour of the game on different graphs.
\subsection{General strategy}

The general principle of the program for solving any of the games is divided into two phases. In the first phase, all final configurations are generated. These configurations are independent of the specific game being played and provide the basis for scoring all games.
In the second phase, the games are played in reverse. 
Once we know the outcomes for the (parent) graphs with $r+1$ coloured edges, we consider all children by uncolouring one edge (of the correct colour when playing the game backwards) and remember the optimal score among all parent graphs that lead to this child.
% For each such graph with $r$ coloured edges, the player goes over all possibilities and the associated score is the optimal one among the scores of the plausible parent graphs. 
This process generates all scores of the graphs with $r$ coloured edges. Continuing the process until the graph contains no coloured edges (or as many as in the initial configuration) yields the final outcome of the game.

\subsubsection{Strategies for efficiency in stage one}

Apart from the data structures (which we will discuss in detail in~\cref{sec:datastructures}) a few techniques are applied to generate the final configurations in an efficient manner. Firstly if the game is played on a clique% and player $1$ starts with one edge and player $2$ can colour $k$ edges
, we can use a peer-reviewed graph generator from Nauty~\cite{Nauty}. Observe that if we only observe final configurations in this game on $n$ vertices, these have $\ceilfrac{n(n-1)}{4}$ red edges. %\ceilfrac{\binom n2}{2}$, $\ceilfrac{n(n-1)}{2(1+k)}$  if the second player can colour k edges a time
We can thus generate all non-isomorphic graphs of order $n$ with $\ceilfrac{n(n-1)}{4}$ edges and assume all edges that are not in this graph to be coloured blue.

This is not the case when either the game is not played on a clique, or player $1$ can start by selecting multiple edges, $E(S)$, on their first turn. For those games a specialised generator was constructed (available in the software package at \cite{github-repository}, in the file \textit{gen-colex.c}). This generator starts from a specific start-graph $S$, representing the edges selected by player $1$ on the first turn, and a base-graph $B$, representing the underlying graph on which the game is played. Here we assume that $E(S) \subseteq E(B)$. Then $\ceilfrac{\binom n2 - \abs{E(S)}}{2}$ edges are placed incrementally, as seen in the function \textit{generate\_final\_configuration}. Starting with $S$, we use Nauty to compute the edge-orbits of the graph, the exact procedure to do this correctly will be explained in detail in \cref{sec:canonical-labelling}. Then for each orbit only the edge with the lowest index in $S$ is added, as seen in the function \textit{compute\_possible\_edges}, and the index of this edge is stored with the graph. The reader can verify this in the function \textit{determine\_canonical\_labelling}. Simultaneously we use Nauty to determine the canonical labelling (for details see \cref{sec:canonical-labelling}) of this resulting graph and temporarily store this to verify that no isomorphic copies are generated. When all non-isomorphic graphs are generated, the process is repeated on all of these graphs until the required number of edges is reached. An important detail is that by storing the index of the last edge added, we can reduce the amount of isomorphic copies generated. To do this, we allow new edges to be added only with an index higher than the last, ensuring a unique way to add a specific set of edges. This trick is seen in the function \textit{compute\_possible\_edges}. This concludes the basic idea of the generator.

\subsubsection{Strategies for efficiency in stage two}

In the second stage, the game is played in reverse. This phase is performed in \textit{erdos-solver.c}. To do this, we need all possible final configurations and their scoring (the scoring only needs to adhere to the principle that every game where player $1$ wins has a higher value than every game where player $2$ wins, though more nuanced scoring can also be used), the base-graph $B$ and start-graph $S$. An iterative algorithm similar to the generator is used, seen by the main structure of the function \textit{find\_best\_game}. This time however every edge of each colour is removed in turn. In a single iteration we start from a list of non-isomorphic graphs, with sets of edges coloured, either in blue or red, and a set of non-coloured edges. Then for each graph every edge in the colour of the player that is currently playing and that is not in $S$ is removed to generate a graph that represents a possible previous state of the game. Using the techniques described in~\cref{sec:canonical-labelling} a canonical representation of the new graph is computed to compare with the graphs already generated in this layer. This is implemented in the function \textit{determine\_canonical\_labelling}. Each time a graph is generated that is an isomorphic copy of one that was already generated, only one copy of the graph is kept (with the optimal score). The optimal score is the maximum in the turn of player $1$ and the minimum in the turn of player $2$. After all possible previous game states have been generated, the process repeats. This iterating stops when only $S$ is found. Then $S$ contains the score that represents the best game player $1$ could play if player $2$ wants to minimise the score of player $1$.

\subsubsection{Datastructures for memory efficiency}
\label{sec:datastructures}

To maximise the memory efficiency, the graphs are represented as bitsets. As the edges of the graphs are the most important for our purposes we encode a graph $G$ as follows: label all vertices from $0$ to $|V(G)|-1$, then label the edge in colexicographic order; $(1,0)$ as $0$, $(2,0)$ as $1$, $(2,1)$ as $2$, $(3,0)$ as $3$, and so on. We now represent each edge with two bits: $00$ represents that the edge is not present in the base-graph and can thus not be coloured, $01$ represents that the edge is not coloured but is present in the base graph and could be coloured, $11$ represents that the edge is present in the base graph, is coloured and has the colour of player $1$ and finally $10$ represents that the edge is present in the base graph, is coloured and has the colour of player $2$. Thus all information of the graph of order $n$ is encoded in a bitset of $n(n-1)$ bits.

\subsubsection{Determining canonical labelling of a coloured graph}
\label{sec:canonical-labelling}

To determine the canonical labelling and the edge orbits of a coloured graph we
% will use the functionality provided by 
will use Nauty (version 2.8.6) \cite{Nauty}. This library provides a functionality that, given the adjacency matrix of an undirected simple graph with given vertex colour classes, efficiently computes a canonical labelling of the graph and refines the vertex orbits based on the graph structure. The reader can see that the encoding presented in the previous~\cref{sec:datastructures} does not comply with this data structure. That is why we expand the graph encoding as follows: Every vertex in the original graph, labelled from $0$ to $|V(G)|-1$, is present in the expansion. Then $\frac{n(n-1)}{2}$ new vertices are inserted, each representing a possible edge in the original encoding. These are ordered in colexicographic order as described above and labelled from $|V(G)|$ up until $|V(G)|+\frac{n(n-1)}{2}-1$. Lastly four additional vertices are inserted, each representing one of the four possible states of the edge. The vertices representing vertices in the original graph are connected to the vertices representing their respective incident edges in the original graph. The vertices representing edges in the original graph are each connected to the correct vertex representing their state in the original graph encoding. Then to ensure vertices are only mapped to vertices, edge to edges and the states to themselves they are each put into their own colour group. The reader can see this expansion in both the generator and solver in the function \textit{generate\_expanded\_graph}. The library is called to canonically label this expanded graph which also returns the vertex orbits for all vertices. We then use the vertex relabelling to construct the bitset of the canonical encoding and store this to compare with other graphs. This is again seen in both implementations in the function \textit{determine\_canonical\_labelling}. Additionally, the vertex orbits in the expanded graph of the vertices that represent the edges in the original graph give us the edge-orbits in the original graph. This last part is present in the function \textit{compute\_possible\_edges} in the file \textit{gen-colex.c}.

\subsection{Proof of program correctness}

% It is evident that the starting vertex is the only non-isomorphic copy of itself.
% Assume that we start with all non-isomorphic graphs with $k$ edges which where created with edges with the minimal index. To prove that we then create all non-isomorphic graphs with $k+1$ edges which where created with edges with the minimal index we assume the contrary: There exists a graph $G$ that is not generated by the algorithm that has $k+1$ edges with edges with the minimal index. Then there exists a subgraph $G'$ without the edge $e_h$ with the highest index of $G$. Then either $G'$ was in the set of graphs provided at the start of this iteration, or it was not. In the first case $G'$ was extend with each edge in its orbit with minimal index and thus either the edge was not minimal in it's orbit, which contradicts the choice of the edge removed in $G$ to create $G'$. Alternatively the edge was minimal but there was another graph isomorphic to $G$ constructed before it from another graph $G''$ in the starting set. However this contradicts the existence of $G$ as the graph generated by adding this other edge to $G''$ would result in a graph with edges with a lower index.

In this subsection, we describe how the program generates all intermediate colourings in accordance with Zermelo’s theorem (readers familiar with its proof may skip this section).

To prove that the program is correct, we first show that the principle of the generator is sound. We do so by induction on the number of red edges in the graph.

It is evident that the starting vertex is the only non-isomorphic copy of itself. Assume we begin with all non-isomorphic graphs with \( k \) edges, created using edges with the minimal index. To prove that we then generate all non-isomorphic graphs with \( k+1 \) edges, created using edges with the minimal index, we assume the contrary: there exists a graph \( G \) that is not generated by the algorithm but has \( k+1 \) edges, all created with edges of minimal index. 

Now, there must exist a subgraph \( G' \) of \( G \) that does not contain the edge \( e_h \) with the highest index in \( G \). There are two possibilities for \( G' \): either it was part of the set of graphs generated at the start of this iteration, or it was not. 

In the first case, \( G' \) would have been extended with each edge in its orbit, using the minimal index. Thus, either the edge was not minimal in its orbit, which contradicts the choice of the edge removed in \( G \) to create \( G' \), or the edge was minimal, but there was another graph isomorphic to \( G \) that had already been constructed from another graph \( G'' \) in the starting set. This would contradict the existence of \( G \), as the graph generated by adding this other edge to \( G'' \) would have resulted in a graph with edges of a lower index.

If $G'$ was not in the starting set, then there is a graph $G''$ that is isomorphic to $G'$ with edges with a lower index than $G'$. Since $G''$ has only $k$ edges and is isomorphic to $G'$, there exists another edge with smallest index that is not in any edge orbits of edges in $G''$ but that is in the edge orbit of one of the edges in $G$. Adding that edge to $G''$ would result in a graph isomorphic to $G$ but with edges with lower index, thus contradicting the choice of $G$.
This proves no such graph can exist and thus that the strategy results in all non-isomorphic graphs with $k+1$ edges with minimal index.

By induction we proved that the generator generates all non-isomorphic graphs with $\ceilfrac{\binom n2 - \abs{E(S)}}{1+k}$ edges.
Next, we prove that the stage-two solver computes the score of the game assuming both players play optimally. We show that the path chosen is optimal by induction on the number of edges to add. Assume that all non-isomorphic final configurations are generated by the generator. Then, each end-configuration is assigned the score corresponding to the optimal outcome that can be reached from it.

Next, we assume that we are given all non-isomorphic graphs after $k+1$ moves with the score of the final configuration reached by an optimal game starting from that graph. Assume that the next player is player $2$, then the current edge was placed by player $1$. The program uncolours every possible red edge from every graph. If two graphs generated this way are isomorphic, the highest score is stored as the score of the optimal game. This is equivalent to selecting the best possible next graph in each possible graph after $k$ moves. Analogously if the current player is player $2$, the program uncolours every blue edge and stores the lowest possible score from a given graph after $k$ moves. This is equivalent to selecting the edge that leads to the worst game for player $1$.

The process converges to the starting graph, which consequently attains the score of the optimal game.

% The reader can see that this process converges on the starting graph whereupon the starting graph will have the score of the optimal game.

\section*{Acknowledgement}

The authors thank David Conlon for referring to~\cite{ES72}, Davide Mattiolo for discussions on $\Clique_{(1,q)}(n)$, Ben Seamone for an idea that lead to Proposition~\ref{prop:(1,2)}, and
Jan Goedgebeur, Jorik Jooken and especially Jarne Renders for discussions and advice on the efficient implementation of the game.

\end{document}